\documentclass[12pt,reqno]{amsart}
\usepackage{amsmath,amsthm,amssymb,amsfonts,amscd}
\usepackage{mathrsfs}
\usepackage{bbm}
\usepackage{bbding}
\usepackage{hyperref}
\usepackage{geometry}
\usepackage{color}
\usepackage{xcolor}

\usepackage{picture,epic}
\usepackage{tikz}

\numberwithin{equation}{section}

\theoremstyle{plain}
\newtheorem{theorem}{Theorem}
\newtheorem{lemma}[theorem]{Lemma}

\theoremstyle{definition}

\theoremstyle{remark}
\newtheorem{remark}{Remark}

\renewcommand{\Re}{\operatorname{Re}}
\renewcommand{\Im}{\operatorname{Im}}

\newcommand{\supp}{\operatorname{supp}}

\newcommand{\Sym}{\operatorname{Sym}}

\newcommand{\GL}{\operatorname{GL}}
\newcommand{\SL}{\operatorname{SL}}
\newcommand{\PSL}{\operatorname{PSL}}

\newcommand{\dd}{\mathrm{d}}
\newcommand{\Res}{\operatorname{Res}}
\newcommand{\area}{\operatorname{area}}

\newcommand{\rmmod}{\operatorname{mod}\;}

\makeatletter
\def\@tocline#1#2#3#4#5#6#7{\relax
  \ifnum #1>\c@tocdepth 
  \else
    \par \addpenalty\@secpenalty\addvspace{#2}%
    \begingroup \hyphenpenalty\@M
    \@ifempty{#4}{%
      \@tempdima\csname r@tocindent\number#1\endcsname\relax
    }{%
      \@tempdima#4\relax
    }%
    \parindent\z@ \leftskip#3\relax \advance\leftskip\@tempdima\relax
    \rightskip\@pnumwidth plus4em \parfillskip-\@pnumwidth
    #5\leavevmode\hskip-\@tempdima
      \ifcase #1
       \or\or \hskip 1em \or \hskip 2em \else \hskip 3em \fi%
      #6\nobreak\relax
    \hfill\hbox to\@pnumwidth{\@tocpagenum{#7}}\par
    \nobreak
    \endgroup
  \fi}
\makeatother

\begin{document}

\title[Quantum variance for Eisenstein Series]
{Quantum variance for Eisenstein Series}
\author{Bingrong Huang}
\address{School of Mathematical Sciences \\ Tel Aviv University \\ Tel Aviv \\ Israel}
\email{bingronghuangsdu@gmail.com}

\date{\today}

\begin{abstract}
  In this paper, we prove an asymptotic formula for the quantum variance for Eisenstein series on $\PSL_2(\mathbb{Z})\backslash \mathbb{H}$. The resulting quadratic form is compared with the classical variance and the quantum variance for cusp forms. They coincide after inserting certain subtle arithmetic factors, including the central values of certain L-functions.
\end{abstract}

\keywords{Quantum variance, Eisenstein series, Hecke L-function, twisted second moment}

\thanks{The work was supported  by the European Research Council, under the European Union's Seventh Framework Programme (FP7/2007-2013)/ERC grant agreement n$^{\text{o}}$~320755.}

\maketitle

\section{Introduction} \label{sec:Intr}


An important problem in the theory of quantum chaos is to understand the variance of matrix coefficients of observables. In the generic chaotic case, there are conjectures in the physics literature relating this quantum variance to the auto-correlation of the observable along the classical motion. These conjectures are wide open except for a handful of cases, namely the modular domain  described below, and the quantized cat map (Kurlberg--Rudnick \cite{KurlbergRudnick2005}).
In this paper, we compute a new instance of the quantum variance, for the continuous spectrum of the modular domain.

Let $G=\PSL_2(\mathbb{R})$, $\Gamma=\PSL_2(\mathbb{Z})$, $\mathbb{H}$ be the upper half plane, and
$\mathbb X=\Gamma\backslash \mathbb{H}$ the modular surface.
$\mathbb{X}$ is a hyperbolic surface that is not compact but of finite area.
The spectrum of the Laplacian on $\mathbb{X}$ has both  discrete and continuous components.
The continuous spectrum is spanned by Eisenstein series.
The discrete spectrum consists of the constants and the space $L_{\rm cusp}^2(\mathbb{X})$ of cusp forms,
for which we can take an orthonormal basis $\{\phi_j\}_{j\in\mathbb{N}}$ of Hecke--Maass forms.
They are real valued and satisfy
$\Delta \phi_j = \lambda_j \phi_j$ and $T_n \phi_j=\lambda_j(n) \phi_j$, where $\Delta$ is the Laplace operator, $T_n$ is the $n$-th Hecke operator. Write $\lambda_j=1/4+t_j^2$.
$\mathbb X$ carries a further symmetry induced by the orientation reversing isometry $z\rightarrow -\bar{z}$ of $\mathbb{H}$ and our $\phi_j$'s are either even or odd with respect to this symmetry.

The goal of this paper is to study the quantum variance for the continuous spectrum, which is parametrized by Eisenstein series.
The Eisenstein series $E(z,s)$ is defined by
\[
  E(z,s) = \sum_{\gamma\in \Gamma_\infty\backslash \Gamma} \Im(\gamma z)^s = \frac{1}{2} \sum_{(c,d)=1} \frac{y^s}{|cz+d|^{2s}},
\]
for $\Re(s)>1$, and has a meromorphic continuation to $s\in\mathbb{C}$, where 
$\Gamma_\infty=\left\{\left(\begin{smallmatrix}1&n\\ &1\end{smallmatrix}\right):n\in\mathbb{Z}\right\}$.
Write $E_{it}(z)=E(z,1/2+it)$.
For a test function $\psi:\mathbb{X}\rightarrow\mathbb{C}$, define
\[
  \mu_t(\psi) = \langle \psi,|E_{it}|^2 \rangle = \int_{\mathbb{X}} \psi(z) |E_{it}(z)|^2 \dd \mu(z),
\]
where $\dd\mu(z)=\frac{\dd x \dd y}{y^2}$.
In this paper, we are interested in the fluctuation of $\mu_t$.
In \cite{Luo-Sarnak1995}, Luo--Sarnak  proved an analogue of quantum unique ergodicity (QUE) for Eisenstein series: for fixed $\psi$ (smooth and compactly-supported on $\mathbb{X}$)
\[
  \mu_t(\psi) \sim \frac{6}{\pi} \log t \int_{\mathbb{X}} \psi(z) \dd \mu(z).
\]
Here we have corrected the constant in place of $\frac{48}{\pi}$ as stated by \cite{Luo-Sarnak1995}, see 
  \cite[eq. (1.1)]{Spinu2003}.
Moreover, let $\phi$ be an even Hecke--Maass cusp form.
They proved
\[
  \mu_t(\phi) \ll t^{-1/6+\varepsilon} \quad \textrm{for all $\varepsilon>0$, as $t\rightarrow\infty$.}
\]
Let $L(s,\phi)$ be the standard L-function of $\phi$.
The Lindel\"{o}f Hypothesis for $L(s,\phi)$ implies $\mu_t(\phi) \ll t^{-1/2+\varepsilon}$.
In \S \ref{sec:EV}, we prove the following upper bound for the expected value.
\begin{theorem}\label{thm:EV}
If $\phi$ is an even Hecke--Maass cusp form, then we have
\begin{equation*}
  \mathbb{E}(\phi;T) = \frac{1}{T}\int_{T}^{2T} \mu_t(\phi) \dd t = o(T^{-1/2}).
\end{equation*}
Note that for an odd Hecke--Maass cusp form $\phi$, we always have $\mu_t(\phi)=0$.
\end{theorem}


Define the quantum variance for Eisenstein series by
\[
  Q_E(\phi,\psi)
  =
  \lim_{T\rightarrow\infty}  \frac{1}{\log T} \int_{T}^{2T} \big( \mu_t(\phi) - \mathbb{E}(\phi;T) \big)
   \overline{\big(\mu_t(\psi)-\mathbb{E}(\psi;T) \big)} \dd t.
\]
for $\phi,\psi\in L_{\rm cusp}^2(\mathbb{X})$.
Note that, by Theorem \ref{thm:EV}, we have
\[
  Q_E(\phi,\psi)
  =
  \lim_{T\rightarrow\infty}  \frac{1}{\log T} \int_{T}^{2T} \mu_t(\phi) \overline{\mu_t(\psi)} \dd t,
\]
Our main result in this paper is as follows.

\begin{theorem}\label{thm:QV}
  Let $\phi,\psi \in \{\phi_j\}_{j\in\mathbb{N}}$ be two Hecke--Maass cusp forms. Then we have
  \[
    Q_E(\phi,\psi)
    = \left\{
    \begin{array}{ll}
        C(\phi) L(\frac{1}{2},\phi)^2 \, V(\phi), & \textrm{if $\phi=\psi$ is even}, \\
      0, & otherwise,
    \end{array} \right.
  \]
  where $C(\phi)$ is a product of local densities given by \eqref{eqn:C} and \eqref{eqn:H(s)}, and
  \begin{equation}\label{eqn:V(phi)}
  V(\phi) = \frac{\big|\Gamma(\frac{1}{4}+\frac{it_\phi}{2})\big|^4} {2\pi |\Gamma(\frac{1}{2}+it_\phi)|^2}.
\end{equation}
  In particular, $Q_E$ is diagonalized by the orthonormal basis $\{\phi_j\}$ of Hecke--Maass cusp forms.
\end{theorem}

\begin{remark}
  In \S \ref{sec:QVweighted}, we give an analogous result for a weighted quantum variance for Eisenstein series. 
  We have the same arithmetic factor $L(\frac{1}{2},\phi)^2$ as in our main theorem, but with another absolute constant instead of $C(\phi)$. This weighted quantum variance can be viewed as a consequence of the asymptotic formula for the second 
  moment of L-functions. To prove our main theorem, we essentially need an asymptotic formula for the twisted second moment.
\end{remark}

\begin{remark}
   It is known \cite{IwaniecSarnak2000} that at least 50\% of even $\phi$'s have $L(1/2,\phi)\neq0$.
   For $\Gamma=\PSL_2(\mathbb{Z})$, one expects $L(1/2,\phi)\neq0$ for every even Hecke--Maass form $\phi$. It can be shown, using a simplification of the expressions \eqref{eqn:a1} and \eqref{eqn:a2} pointed out by Peter Humphries, that $C(\phi)$ is never zero.
\end{remark}

\medskip

For the purpose of comparison, we introduce the classical variance.
The fluctuations of an observable $\psi\in C_0(\Gamma\backslash G)$ under the geodesic flow $\mathcal{G}_t$ was determined in  \cite{Ratner1973} and \cite{Ratner1987}, and it asserts that as $T$ goes to infinity,
$
  \frac{1}{\sqrt{T}}\int_0^T \psi(\mathcal{G}_t(x))\dd t
$
as a random variable on $ \Gamma\backslash G$ becomes Gaussian with mean zero and variance $V$ given by
\begin{equation}\label{eqn:V}
  V(\psi_1,\psi_2) = \int_{-\infty}^{\infty} \int_{\Gamma\backslash G} \psi_1(\mathcal{G}_t(x)) 
  \overline{\psi_2(x)} \dd x \dd t.
\end{equation}
Note that \eqref{eqn:V} converges due to  the rapid decay of correlations for the geodesic flow.
It has been conjectured in \cite{FP1986,EFKAMM}, that for ``generic'' chaotic systems such as the one at hand, the quantum fluctuations are also Gaussian with a variance which agrees with the classical one in \eqref{eqn:V}.
Let $\phi$ be a fixed even Hecke--Maass cusp form with the Laplace eigenvalue $\lambda_\phi=1/2+it_\phi$.
By \cite[Appendix A.1]{Luo-Sarnak2004}, we have $V(\phi,\phi) = V(\phi)$.
Hence Theorem \ref{thm:QV} asserts that the quantum variance for Eisenstein series is equal to the classical variance after inserting the ``correction factor'' of
$C(\phi) L(\frac{1}{2},\phi)^2$.

\medskip

To shed some light on this correction factor, we introduce the quantum variance for cusp forms,
which measures the fluctuations of the probability measures, 
$\dd \mu_{j}(z) = |\phi_j(z)|^2 \dd\mu(z),$
 in the semi-classical limit $t_j\rightarrow\infty$.
The quantum ergodicity theorem (QE) proved by Shnirelman \cite{Shnirelman1974}, Colin de Verdi\`{e}re \cite{CdV1985}, and Zelditch \cite{Zelditch1987} for compact surfaces, and extended by Zelditch \cite{Zelditch1992} to noncompact surfaces such as our $\mathbb{X}$, implies that there exists a full density subsequence $\lambda_{j_k}$ (i.e. $\sum\limits_{\lambda_{j_k}\leq \lambda} 1 \sim \sum\limits_{\lambda_j\leq \lambda} 1 \sim \frac{\area(\mathbb{X})}{4\pi} \lambda $) such that for an ``observable'' $\psi\in C(\mathbb{X})$,
\[
  \lim_{j_k\rightarrow\infty}\mu_{j_k}(\psi) = \frac{3}{\pi} \int_{\mathbb{X}} \psi(z) \dd \mu(z).
\]
The QUE conjectured by Rudnick--Sarnak \cite{RudnickSarnak1994} and proved by
Lindenstrauss \cite{Lindenstrauss2006} and 
Sounda-rarajan \cite{Soundararajan2010} asserts that there are no exceptional subsequences, i.e., $\dd\mu_j\rightarrow \frac{3}{\pi}\dd\mu$ as $j\rightarrow\infty$.

We can also consider holomorphic cusp forms in $S_k(\Gamma)$ of even integral weight $k$ for $\Gamma$.
$S_k(\Gamma)$ is a finite dimensional Hilbert space. Let $H_k$ be the orthonormal basis of Hecke cusp forms in $S_k(\Gamma)$. We have $\dim S_k(\Gamma) = \# H_k \sim k/12$ as $k\rightarrow\infty$.
QUE for the measures $\dd\mu_f = y^k|f(z)|^2 \dd\mu(z)$ for $f\in H_k$
was proved by Holowinsky--Soundararajan \cite{HS2010}.

\medskip

Let $C_{0,0}^\infty(\mathbb{X})$ be the space of smooth functions (e.g. $\psi$) on $\mathbb{X}$ that decay rapidly in the cusp with mean zero (i.e. $\int_\mathbb{X} \psi(z) \dd\mu(z)=0$) and whose zeroth Fourier coefficient $\int_{0}^{1}\psi(x+iy)\dd x$ is zero for $y$ large enough (depending on $\psi$).
Note that $L_{\rm cusp}^2(\mathbb{X})\subset C_{0,0}^\infty(\mathbb{X})$.
In order to prove ``Shnirelman's theorem'', Zelditch \cite{Zelditch1994} introduced quantum variance sums for cusp forms,
\[
  S_\psi(\lambda) = \sum_{\lambda_j\leq \lambda} |\mu_j(\psi)|^2 \quad
  \textrm{for $\psi\in C_{0,0}^\infty(\mathbb{X})$}.
\]
He showed the non-trivial upper bound $S_\psi(\lambda)\ll_\psi \lambda/\log \lambda$.
Luo--Sarnak \cite{Luo-Sarnak1995} established $S_\psi(\lambda)\ll_\psi \lambda^{1/2+\varepsilon}$ for any $\varepsilon>0$.
In \cite{Luo-Sarnak2004}, Luo--Sarnak proved an asymptotic formula for a weighted  quantum variance sum for holomorphic modular forms.
More precisely,
let $L(s,f)$ and $L(s,\phi)$ be the corresponding standard L-functions, and let $\Sym^2 f$ and $\Sym^2 \phi$ be the symmetric square lifts of $f$ and $\phi$. Let $L(s,\Sym^2 f)$ and $L(s,\Sym^2 \phi)$ be the corresponding L-functions.
For fixed $u\in C_0^\infty(0,\infty)$, let
\[
  B_\omega(\psi) = \lim_{K\rightarrow\infty} \frac{1}{K \int_{0}^{\infty}u(x)\dd x} \sum_{2|k} u\Big(\frac{k-1}{K}\Big) \sum_{f\in H_k} L(1,\Sym^2f) |\mu_f(\psi)|^2,
\]
which is a non-negative Hermitian form defined on $C_{0,0}^\infty(\mathbb{X})$.
Luo--Sarnak showed that $B_\omega$ satisfies the symmetries
\begin{equation*}
  B_\omega(\Delta \psi_1,\psi_2) = B_\omega( \psi_1,\Delta \psi_2), \quad
  B_\omega(T_n \psi_1,\psi_2) = B_\omega( \psi_1,T_n \psi_2) \quad \textrm{for all $n\geq1$}.
\end{equation*}
Restricting $B_\omega$ to $L_{\rm cusp}^2(\mathbb{X})$, $B_\omega$ is diagonalized by $\{\phi_j\}_{j\in\mathbb{N}}$ and the eigenvalue of $B_\omega$ at $\phi_j$ is $\frac{\pi}{2}L(1/2,\phi_j)$.
Similar results for Hecke--Maass cusp forms were proved in Luo--Sarnak \cite{Luo-Sarnak2004} and extended by Zhao \cite{Zhao2010},
the only difference is that the eigenvalue of $B$ at $\phi$ is given by
\begin{equation*}
  B(\phi) = \frac{1}{2} L\Big(\frac{1}{2},\phi\Big)V(\phi).
\end{equation*}
Recently, Sarnak--Zhao \cite{SarnakZhao2013} obtained the asymptotic formula of quantum variance for several phase space observables, that is for
Hecke--Maass cusp forms on $\Gamma\backslash G$.
They also removed the harmonic weights 
for the variance, getting a result with a further positive factor  depending on $\phi$ which is a product of local densities (see \cite[Corollary 1]{SarnakZhao2013}).

\medskip

Note that for the quantum variance of cusp forms we have a factor $L(\frac{1}{2},\phi)$, while for the quantum variance of Eisenstein series we get a factor $L(\frac{1}{2},\phi)^2$.
For a holomorphic cusp form $f$ and an even Hecke--Maass form $\phi$, by Watson's formula \cite{Watson2008}, we have
  \[
    |\mu_f(\phi)|^2 = \frac{\Lambda(1/2,\Sym^2(f)\times \phi) \Lambda(1/2,\phi)}{8\Lambda(1,\Sym^2f)^2 \Lambda(1,\Sym^2\phi)},
  \]
  where $\Sym^2(f)\times \phi$ is the Rankin--Selberg convolution, and $\Lambda$ means the corresponding completed L-functions. After averaging over $f$, this explains the existence of the factor $L(1/2,\phi)$.
  However, for Eisenstein series, the Rankin--Selberg method gives (see \eqref{eqn:RS} and \eqref{eqn:mu_t=})
  \[
    |\mu_t(\phi)|^2 = \frac{|\Lambda(1/2+2it,\phi)|^2 \Lambda(1/2,\phi)^2 } {2|\xi(1+2it)|^4 \Lambda(1,\Sym^2\phi)},
  \]
  where $\xi(s)=\pi^{-s/2}\Gamma(s/2)\zeta(s)$. This will give us the factor $L(\frac{1}{2},\phi)^2$.

  In our case, we can diagonalize $Q_E$ on $L_{\rm cusp}^2(\mathbb{X})$ rather easily compared with $B_\omega$ in \cite{Luo-Sarnak2004}. This thanks to the factor $\log T$, which comes from the second moment of Hecke L-functions. From our theorem, we also have the symmetries
  \begin{equation*}
    Q_E(\Delta \phi,\psi) = Q_E( \phi,\Delta \psi), \quad
    Q_E(T_n \phi,\psi) = Q_E( \phi,T_n \psi)
    \quad \textrm{for all $n\geq1$}.
\end{equation*}

\medskip

Recently, Nelson \cite{Nelson2016,Nelson2017} determined the quantum variance on quaternion algebras.
In \cite{LRS2009}, Luo--Rudnick--Sarnak considered the variance of arithmetic measures associated to closed geodesics on the modular surface. The resulting variance is very close to the quantum variance for Hecke--Maass forms.

\section{Preliminaries}

\subsection{The Rankin--Selberg method}

Let $\phi$ be a fixed even Hecke--Maass cusp form on $\PSL_2(\mathbb{Z})\backslash \mathbb{H}$ with the normalization $\|\phi\|_2=1$.
It has a Fourier expansion
\[
  \phi(z) = \frac{1}{2}\rho_\phi(1) y^{1/2} \sum_{n\neq0} \lambda_\phi(n) K_{it_\phi}(2\pi|n|y) e(nx).
\]
Note that we have $\lambda_\phi(n)\in\mathbb{R}$, and we normalize $\rho_\phi(1)\in\mathbb{R}$.

By Rankin--Selberg convolution (see e.g. \cite[\S7.2]{goldfeld2006automorphic}), we have
\[
  \Res_{s=1} \zeta(2s) \langle \phi^2,E(\cdot,\bar{s}) \rangle
  = \frac{1}{2^4}
  \rho_\phi(1)^2 L(1,\Sym^2\phi) \Big|\Gamma\Big(\frac{1}{2}+it_\phi\Big)\Big|^2.
\]
On the other hand, since $E(z,s)$ has a simple pole at $s=1$ with the residue $3/\pi$, we have
\[
  \Res_{s=1} \zeta(2s) \langle \phi^2,E(\cdot,\bar{s}) \rangle
  = \frac{3}{\pi} \zeta(2) \|\phi\|_2^2
  = \frac{3}{\pi} \zeta(2).
\]
Hence we have
\begin{equation}\label{eqn:RS}
  \rho_\phi(1)^2 L(1,\Sym^2\phi)
  =
  \frac{2^3 \pi } {|\Gamma(\frac{1}{2}+it_\phi)|^2}.
\end{equation}

\subsection{The Riemann zeta function}

Using a zero free region for $\zeta(s)$, we get the following
approximate functional equation for $\frac{1}{\zeta(1+2it)^2}$.
\begin{lemma}\label{lemma:AFE-zeta}
  Let $t\asymp T$ be large enough. Then for any $T^\varepsilon\ll x \ll T^B$, we have
  \[
    \frac{1}{\zeta(1+2it)^2} = \sum_{k\leq x^{1+\varepsilon} } \frac{\alpha(k)}{k^{1+2it}} e^{-k/x} + O(e^{-(\log T)^{1/5}}),
  \]
  where $\alpha(k) = \sum_{mn=k}\mu(m)\mu(n)$.
\end{lemma}

\begin{proof}
  Consider $I = \frac{1}{2\pi i} \int_{(1)} \frac{1}{\zeta(1+2it+s)^2} \Gamma(s) x^s \dd s$. By the Dirichlet series expression of $\frac{1}{\zeta(s)^2}=\sum_{k=1}^{\infty}\frac{\alpha(k)}{k^s}$, we have
  \begin{equation}\label{eqn:I1}
    I 
    = \sum_{k=1}^{\infty} \frac{\alpha(k)}{k^{1+2it}} e^{-k/x}.
  \end{equation}
  On the other hand, we can move the contour to the left, to one along the straight line segments $L_1,L_2,L_3$ defined by
  $$
    L_1=\left\{ u_0+iv: |v|\leq T \right\}, \quad
    L_2=\left\{ 1+iv: |v|\geq T \right\},
  $$
  and the short horizontal segments
  $$
    L_3=\left\{ u\pm iT: u_0 \leq \sigma \leq 1 \right\},
  $$
  where $u_0 = -\frac{c}{(\log T)^{3/4}}$ with $c$ being a small positive number such that $\zeta(1+2it+s)$ is zero-free
  on the boundary and right side of $L_1\cup L_2\cup L_3$.
  By \cite[Theorem 8.29]{iwaniec2004analytic}, we may also use
  the Vinogradov--Korobov bound $1/\zeta(s)\ll (\log |\Im(s)|)^{2/3}(\log\log |\Im(s)|)^{1/3}$ in this region.
  The integrals along the line segments $L_2$ and $L_3$ are trivially bounded by $O(T^{-100})$ by the rapid decay of $\Gamma(s)$.
  The new line $L_1$ gives an amount that is certainly
  \begin{equation*}
    \ll \log T \exp\left( -\frac{\log x}{(\log T)^{3/4}} \right)
    \ll \exp(-(\log T)^{1/5}).
  \end{equation*}
  Now we need to analyze the residue of the pole from $\Gamma(s)$.
  The residue at $s=0$ contributes $\frac{1}{\zeta(1+2it)^2}$.
  Thus we have
  \begin{equation}\label{eqn:I2}
    I = \frac{1}{\zeta(1+2it)^2} + O(e^{-(\log T)^{1/5}}).
  \end{equation}
  By \eqref{eqn:I1} and \eqref{eqn:I2}, we complete the proof of our lemma.
\end{proof}

\subsection{Hecke L-functions}

Let $\phi,\psi\in\{\phi_j\}_{j\in\mathbb{N}}$ be two even Hecke--Maass cusp forms.
Recall that we have the functional equation for $L(s,\phi)$,
\[
  \Lambda(s,\phi) = L_\infty(s,\phi) L(s,\phi) = \Lambda(1-s,\phi),
\]
where $L_\infty(s,\phi)=\pi^{-s} \Gamma(\frac{s+it_\phi}{2})\Gamma(\frac{s-it_\phi}{2})$,
and similarly for $\psi$.
We will use the following approximate functional equation. See e.g. \cite[\S5.2]{iwaniec2004analytic} for more details.

\begin{lemma}\label{lemma:AFE-Lfcn}
  For $t\asymp T$ large, we have
  \begin{multline*}
    L\Big(\frac{1}{2}-2it,\phi\Big) L\Big(\frac{1}{2}+2it,\psi\Big) \\
    = \sum_{m,n} \frac{\lambda_\phi(m)\lambda_\psi(n)}{\sqrt{mn}} \Big(\frac{m}{n}\Big)^{2it} W_t(mn)
    +
    \sum_{m,n} \frac{\lambda_\phi(m)\lambda_\psi(n)}{\sqrt{mn}} \Big(\frac{n}{m}\Big)^{2it} W_t(mn) + O(T^{-1+\varepsilon}),
  \end{multline*}
  where
  \[
    W_t(y) = \frac{1}{2\pi i} \int_{(2)} \Big(\frac{t^2}{\pi^2 y} \Big)^s e^{s^2} \frac{\dd s}{s}.
  \]
\end{lemma}

\begin{proof}
  Consider $\frac{1}{2\pi i} \int_{(2)} \frac{\Lambda(\frac{1}{2}-2it+s,\phi)\Lambda(\frac{1}{2}+2it+s,\psi)} {L_\infty(\frac{1}{2}-2it,\phi)L_\infty(\frac{1}{2}+2it,\psi)} e^{s^2} \frac{\dd s}{s}$. By shifting the contour to the left, and using the functional equation, we get
  \begin{equation*}
    L\Big(\frac{1}{2}-2it,\phi\Big) L\Big(\frac{1}{2}+2it,\psi\Big)
    \\ =
    \sum_{m,n} \frac{\lambda_\phi(m)\lambda_\psi(n)}{\sqrt{mn}}
    \Big(\sum_{\pm} \Big(\frac{m}{n}\Big)^{\pm 2it} W_t^\pm(mn) \Big),
  \end{equation*}
  where
  \[
    W_t^\pm(y) = \frac{1}{2\pi i} \int_{(2)} y^{-s} \frac{L_\infty(\frac{1}{2}\mp 2it+s,\phi) L_\infty(\frac{1}{2}\pm 2it+s,\psi)} {L_\infty(\frac{1}{2}-2it,\phi)L_\infty(\frac{1}{2}+2it,\psi) } e^{s^2} \frac{\dd s}{s}.
  \]
  Note that Stirling's formula allows us to truncate the sum over $m,n$ with $mn\leq T^{2+\varepsilon}$.
  We can shift the contour to $\Re(s)=\varepsilon$ and  truncate $s$ to $|\Im(s)|\leq T^\varepsilon$ by the rapid decay of $e^{s^2}$.
  Then by Stirling's formula, we get
  \[
    \frac{L_\infty(\frac{1}{2}\mp 2it+s,\phi) L_\infty(\frac{1}{2}\pm 2it+s,\psi)} {L_\infty(\frac{1}{2}-2it,\phi)L_\infty(\frac{1}{2}+2it,\psi) }
    = \pi^{-2s} t^{2s} \big(1+O(T^{-2+\varepsilon})\big).
  \]
  Now the lemma follows easily.
\end{proof}

\section{Weighted quantum variance}\label{sec:QVweighted}

In this section, we will prove the following result for a weighted quantum variance.
\begin{theorem} \label{thm:QVweighted}
  Let $\phi$ be an even Hecke--Maass cusp form. Then we have
  \begin{equation*} 
  \lim_{T\rightarrow\infty} \frac{1}{\log T} \int_{T}^{2T}  |\zeta(1+2it)|^4
  \big| \mu_{t}(\phi) \big|^2  \dd t
  =
  (12 \log 2) \, L\Big(\frac{1}{2},\phi\Big)^2 \, V(\phi) .
  \end{equation*}
\end{theorem}

\begin{proof}
By the Rankin--Selberg method (see \cite[\S2]{Luo-Sarnak1995} and correct the constant), we have
\begin{multline}\label{eqn:mu_t=}
  \mu_t(\phi)  =
  \frac{\rho_\phi(1)}{4} \frac{\Lambda(\frac{1}{2},\phi)\Lambda(\frac{1}{2}-2it,\phi)} {|\xi(1+2it)|^2}
  \\
   = \frac{\rho_\phi(1)}{4}  \pi^{2it}
  \frac{\big|\Gamma(\frac{1}{4}+\frac{it_\phi}{2})\big|^2 \Gamma(\frac{1}{4}-\frac{it_\phi}{2}-it) \Gamma(\frac{1}{4}+\frac{it_\phi}{2}-it)} {\big|\Gamma(\frac{1}{2}+it)\big|^2}
  \frac{L\big(\frac{1}{2},\phi\big) L\big(\frac{1}{2}-2it,\phi\big)}{|\zeta(1+2it)|^2} .
\end{multline}
Note that $\rho_\phi(1)$ is real and $L(1/2,\phi)\geq0$ \cite{KatokSarnak1993}. We have
\begin{multline*}
  \frac{1}{T} \int_{T}^{2T}  |\zeta(1+2it)|^4
  \big|\mu_t(\phi) \big|^2  \dd t
  = \frac{\rho_\phi(1)^2 }{2^4} \Big|\Gamma\Big(\frac{1}{4}+\frac{it_\phi}{2}\Big)\Big|^4 L\Big(\frac{1}{2},\phi\Big)^2  \\
  \cdot
    \frac{1}{T} \int_{T}^{2T}
    \frac{\big|\Gamma(\frac{1}{4}-\frac{it_\phi}{2}-it) \Gamma(\frac{1}{4}+\frac{it_\phi}{2}-it)\big|^2} {\big|\Gamma(\frac{1}{2}+it)\big|^4}
   \Big|L\Big(\frac{1}{2}-2it,\phi\Big)\Big|^2 \dd t.
\end{multline*}

By Stirling's formula, for $t\sim T$ large and $t_\phi$ fixed, we have
\[
  \frac{\big|\Gamma(\frac{1}{4}-\frac{it_\phi}{2}-it) \Gamma(\frac{1}{4}+\frac{it_\phi}{2}-it)\big|^2} {\big|\Gamma(\frac{1}{2}+it)\big|^4} = \frac{1}{t} + O\Big(\frac{1}{T^2}\Big).
\]
By the second moment of L-functions (see e.g. \eqref{eqn:SM-shortint} below), we know that the contribution from the above error is $O(T^{-2+\varepsilon})$.
To prove our theorem, it suffices to consider
\[
  \frac{1}{T} \int_{T}^{2T}
   \Big|L\Big(\frac{1}{2}-2it,\phi\Big)\Big|^2 \frac{\dd t}{t}.
\]
By 
\cite[Theorem 1]{Kuznetsov1981}, we have
\[
  \frac{1}{T}\int_{0}^{T}
   \Big|L\Big(\frac{1}{2}+it,\phi\Big)\Big|^2 \dd t
   =  \frac{2^4 \cosh(\pi t_\phi)}{\zeta(2) \rho_\phi(1)^2}
   \big( \log T + B_\phi \big) +  O\Big(T^{-1/7+\varepsilon}\Big),
\]
where $B_\phi$ is some constant depending on $\phi$. Hence
\begin{equation}\label{eqn:SM-TK}
  \frac{1}{K}\int_{T_0}^{T_0+K}
   \Big|L\Big(\frac{1}{2}+it,\phi\Big)\Big|^2 \dd t
   =  \frac{2^4 \cosh(\pi t_\phi)}{\zeta(2) \rho_\phi(1)^2}
   \log T + O\big(1\big),
\end{equation}
for $T_0\asymp T$ and $K=T/\log T$. Note that
\[
  \frac{1}{T} \int_{T}^{2T}
   \Big|L\Big(\frac{1}{2}-2it,\phi\Big)\Big|^2 \frac{\dd t}{t}
   = \frac{1}{T} \sum_{j=0}^{[\log T]} \int_{T+jK}^{T+(j+1)K}
   \Big|L\Big(\frac{1}{2}+2it,\phi\Big)\Big|^2 \frac{\dd t}{t}
   +O\Big(\frac{1}{T}\Big).
\]
By \eqref{eqn:SM-TK} and the partial summation, we get
\[
  \frac{1}{T} \int_{T}^{2T}
   \Big|L\Big(\frac{1}{2}-2it,\phi\Big)\Big|^2 \frac{\dd t}{t}
   = \frac{2^4 \cosh(\pi t_\phi)}{\zeta(2) \rho_\phi(1)^2}  (\log 2) \frac{1}{T} \log T + O\Big(\frac{1}{T}\Big).
\]
Thus by \eqref{eqn:V(phi)}, we have
\begin{equation*}
  \frac{1}{\log T} \int_{T}^{2T}  |\zeta(1+2it)|^4
  \big| \mu_t(\phi) \big|^2  \dd t
  \\
  \sim
  (12 \log 2) \, L\Big(\frac{1}{2},\phi\Big)^2 \, V(\phi),
\end{equation*}
as claimed.
\end{proof}


\begin{remark}
  If $\phi,\psi \in \{\phi_j\}_{j\in\mathbb{N}}$ are two distinct Hecke--Maass cusp forms, then by the method of the proof of Theorem \ref{thm:QV}, we can show that
  \begin{equation*} 
  \lim_{T\rightarrow\infty} \frac{1}{\log T} \int_{T}^{2T}  |\zeta(1+2it)|^4
  \mu_t(\phi) \overline{\mu_t(\psi)}  \dd t
  =
  0.
  \end{equation*}
\end{remark}

\section{Quantum variance: setup} \label{sec:QVsetup}

In this section, we will start to prove Theorem \ref{thm:QV}.
%
As in the previous section, we have
\begin{multline} \label{eqn:QV=L}
  \frac{1}{T} \int_{T}^{2T}
  \mu_t(\phi) \overline{\mu_t(\psi)}  \dd t
  = \frac{\rho_\phi(1) \rho_\psi(1)}{2^4} \Big|\Gamma\Big(\frac{1}{4}+\frac{it_\phi}{4}\Big)\Big|^2 \Big|\Gamma\Big(\frac{1}{4}+\frac{it_\psi}{4}\Big)\Big|^2 L\Big(\frac{1}{2},\phi\Big)  L\Big(\frac{1}{2},\psi\Big) \\
  \cdot
    \frac{1}{T} \int_{T}^{2T}
    \frac{\Gamma(\frac{1}{4}-\frac{it_\phi}{2}-it) \Gamma(\frac{1}{4}+\frac{it_\phi}{2}-it) \Gamma(\frac{1}{4}-\frac{it_\psi}{2}+it) \Gamma(\frac{1}{4}+\frac{it_\psi}{2}+it)} {\big|\Gamma(\frac{1}{2}+it)\big|^4 } \\
  \cdot \frac{L(\frac{1}{2}-2it,\phi) L(\frac{1}{2}+2it,\psi)}{|\zeta(1+2it)|^4} \dd t \\
  = \frac{\rho_\phi(1) \rho_\psi(1)}{2^4} \Big|\Gamma\Big(\frac{1}{4}+\frac{it_\phi}{4}\Big)\Big|^2 \Big|\Gamma\Big(\frac{1}{4}+\frac{it_\psi}{4}\Big)\Big|^2 L\Big(\frac{1}{2},\phi\Big)  L\Big(\frac{1}{2},\psi\Big)
  \\ \cdot
   \frac{1}{T} \int_{T}^{2T} \frac{L(\frac{1}{2}-2it,\phi) L(\frac{1}{2}+2it,\psi)}{|\zeta(1+2it)|^4}
    \frac{\dd t}{t} + O(T^{-2+\varepsilon}) .
\end{multline}
To prove an asymptotic formula, it suffices to consider
\begin{equation}\label{eqn:LLaver}
  \frac{1}{T} \int_{T}^{2T} \frac{L(\frac{1}{2}-2it,\phi) L(\frac{1}{2}+2it,\psi)} {|\zeta(1+2it)|^4}
    \frac{\dd t}{t}.
\end{equation}
To do this, we may insert a smooth weight for the integral, so that it suffices to deal with
\begin{equation}\label{eqn:LLaver-w}
  \frac{1}{T^2} \int_{0}^{\infty} w\Big(\frac{t}{T}\Big) \frac{L(\frac{1}{2}-2it,\phi) L(\frac{1}{2}+2it,\psi)} {|\zeta(1+2it)|^4} \dd t,
\end{equation}
where $w(y)=w_U(y)$ is a smooth function with support in the interval $[1-1/U,2+1/U]$, $U\geq2$, such that $w_U^{(j)}(y) \ll U^j$ for $j\geq0$ and
$w_U(y)=1/y$ for $y\in[1+1/U,2-1/U]$.
Here we assume $T^\varepsilon\leq U \leq T^{1/3}$.
By the upper bound for the second moment of Hecke L-functions
(see e.g. \cite{good1982} and \cite{Jutila})
\begin{equation}\label{eqn:SM-shortint}
  \int_{T}^{T+M} \Big|L\Big(\frac{1}{2}+2it,\phi\Big)\Big|^2 \dd t \ll_\phi MT^\varepsilon,
\end{equation}
for $T^{2/3} \leq M \leq T$, we know
the difference between \eqref{eqn:LLaver} and \eqref{eqn:LLaver-w} is bounded by $O(T^{\varepsilon-1} U^{-1})$.
By Lemma \ref{lemma:AFE-zeta}, it then suffices to consider
\begin{equation}\label{eqn:I=}
  \mathcal I=\frac{1}{T^2} \int_{0}^{\infty} w\Big(\frac{t}{T}\Big)
  \Big|\sum_{k\leq x^{1+\varepsilon}} \frac{\alpha(k)}{k^{1+2it}}e^{-k/x}\Big|^2
  L\Big(\frac{1}{2}-2it,\phi\Big)
  L\Big(\frac{1}{2}+2it,\psi\Big) \dd t,
\end{equation}
for some $x=T^\delta$, with $\delta>0$ a fixed small number. By Lemma \ref{lemma:AFE-Lfcn}, we get
\begin{multline*}
  \mathcal{I}
  = \frac{1}{T^2} \int_{0}^{\infty} w\Big(\frac{t}{T}\Big)  \Big|\sum_{k\leq x^{1+\varepsilon}} \frac{\alpha(k)}{k^{1+2it}}e^{-k/x}\Big|^2
  \\ \cdot
   \Big( \sum_{m,n} \frac{\lambda_\phi(m)\lambda_\psi(n)}{\sqrt{mn}} \sum_\pm \Big(\frac{n}{m}\Big)^{\pm 2it} W_t(mn) \Big)
   \dd t + O(T^{-2+\varepsilon})
   \\ =
   \frac{1}{T^2} \int_{0}^{\infty} w\Big(\frac{t}{T}\Big)
   \sum_{k\leq x^{1+\varepsilon}}
   \sum_{\ell\leq x^{1+\varepsilon}} \frac{\alpha(k)\alpha(\ell)}{k\ell} e^{-k/x}e^{-\ell/x}
   \sum_{m,n} \frac{1}{\sqrt{mn}} \Big(\frac{\ell n}{km}\Big)^{2it} \\
   \cdot
   \big(\lambda_\phi(m)\lambda_\psi(n) +\lambda_\psi(m)\lambda_\phi(n)\big) W_t(mn) \dd t
   + O(T^{-2+\varepsilon}) .
\end{multline*}
Hence we have
\begin{equation}\label{eqn:I=D+O}
  \mathcal{I}  = \mathcal{D}_{\phi,\psi} + \mathcal{D}_{\psi,\phi} + \mathcal{O}_{\phi,\psi} +  \mathcal{O}_{\psi,\phi} + O(T^{-2+\varepsilon}),
\end{equation}
where $\mathcal{D}_{\phi,\psi}$ is the diagonal terms
\[
  \mathcal D_{\phi,\psi} = \frac{1}{T^2} \int_{0}^{\infty} w\Big(\frac{t}{T}\Big)
   \sum_{k,\ell} \frac{\alpha(k)\alpha(\ell)}{k\ell} e^{-k/x}e^{-\ell/x}
   \sum_{\substack{m,n\\ \ell n=km}} \frac{\lambda_\phi(m)\lambda_\psi(n)}{\sqrt{mn}} W_t(mn)
   \dd t,
\]
and
$\mathcal{O}_{\phi,\psi}$ is the off-diagonal terms
\[
  \mathcal O_{\phi,\psi} = \frac{1}{T^2} \int_{0}^{\infty} w\Big(\frac{t}{T}\Big)
   \sum_{k,\ell} \frac{\alpha(k)\alpha(\ell)}{k\ell} e^{-k/x}e^{-\ell/x}\sum_{h\neq0}
   \sum_{\substack{m,n\\ km-\ell n=h}} \frac{\lambda_\phi(m)\lambda_\psi(n)}{\sqrt{mn}}
   \Big(\frac{\ell n}{km}\Big)^{2it} W_t(mn)
   \dd t .
\]

\section{Diagonal terms}

In this section, we deal with the diagonal terms $\mathcal{D}_{\phi,\psi}$.
We first consider the sums,
\[
  \mathcal{S}_{\phi,\psi} = \sum_{k,\ell} \frac{\alpha(k)\alpha(\ell)}{k\ell} e^{-k/x}e^{-\ell/x}
   \sum_{\substack{m,n\\ \ell n=km}} \frac{\lambda_\phi(m)\lambda_\psi(n)}{\sqrt{mn}} W_t(mn).
\]
Using the integral expressions of the weight functions, we get
\begin{multline}\label{eqn:S=int}
  \mathcal S_{\phi,\psi} = \frac{1}{(2\pi i)^3} \int_{(2)} \int_{(2)} \int_{(2)} \sum_{k,\ell} \frac{\alpha(k)\alpha(\ell)}{k^{1+s_1}\ell^{1+s_2}}
   \sum_{\substack{m,n\\ \ell n=km}} \frac{\lambda_\phi(m)\lambda_\psi(n)}{(mn)^{1/2+s}}
   \\ \cdot
   x^{s_1+s_2} \Big(\frac{t^2}{\pi^2} \Big)^s \Gamma(s_1) \Gamma(s_2)\frac{ e^{s^2}}{s} \dd s \dd s_1 \dd s_2.
\end{multline}
We compute the Dirichlet series
\[
  D_{\phi,\psi}(s,s_1,s_2) = \sum_{k,\ell} \frac{\alpha(k)\alpha(\ell)}{k^{1+s_1}\ell^{1+s_2}}
   \sum_{\substack{m,n\\ \ell n=km}} \frac{\lambda_\phi(m)\lambda_\psi(n)}{(mn)^{1/2+s}}.
\]
Let $d=\gcd(k,\ell)$, $k=k'd$, $\ell=\ell' d$, so that $\ell'n=k'm$. Thus $\ell'\mid m$, $k'\mid n$, and $m/\ell'=n/k'$.
By changing variables, we get
\[
  D_{\phi,\psi}(s,s_1,s_2) = \sum_{\substack{k,\ell,d \\ (k,\ell)=1}} \frac{\alpha(kd)\alpha(\ell d)}{ k^{3/2+s+s_1}\ell^{3/2+s+s_2}d^{2+s_1+s_2} }
   \sum_{n} \frac{\lambda_\phi(\ell n) \lambda_\psi(k n)}{n^{1+2s}}.
\]
By Appendix \ref{app:DS}, we have
\begin{equation}\label{eqn:D=LH}
  D_{\phi,\psi}(s,s_1,s_2) = \frac{L(1+2s,\phi\times \psi)}{\zeta(2+4s)} H_{\phi,\psi}(s,s_1,s_2),
\end{equation}
where
\begin{multline}\label{eqn:H(s)}
  H_{\phi,\psi}(s,s_1,s_2) = \prod_p  \Big( 1 + \frac{4}{p^{2+s_1+s_2}} + \frac{1}{p^{2(2+s_1+s_2)}}  \\ + \frac{a_{\phi,\psi,1}(p)}{p^{3/2+s+s_1}} + \frac{a_{\phi,\psi,2}(p)}{p^{2(3/2+s+s_1)}} + \frac{a_{\psi,\phi,1}(p)}{p^{3/2+s+s_2}} +
  \frac{a_{\psi,\phi,2}(p)}{p^{2(3/2+s+s_2)}} \Big),
\end{multline}
with
\begin{multline}\label{eqn:a1}
  a_{\phi,\psi,1}(p) = \Big(\sum_{j=0}^{\infty} \frac{\lambda_\phi(p^{j})\lambda_\psi(p^{j+1})}{p^{j(1+2s)}}\Big) \Big(-2-\frac{2}{p^{2+s_1+s_2}}\Big)   \\ \cdot \Big(1-\frac{\lambda_\phi(p)\lambda_\psi(p)}{p^{1+2s}} +\frac{\lambda_\phi(p^2)+\lambda_\psi(p^2)}{p^{2(1+2s)}} -\frac{\lambda_\phi(p)\lambda_\psi(p)}{p^{3(1+2s)}} +\frac{1}{p^{4(1+2s)}}\Big) \Big(1-\frac{1}{p^{2(1+2s)}}\Big)^{-1},
\end{multline}
\begin{multline}\label{eqn:a2}
  a_{\phi,\psi,2}(p) = \Big(\sum_{j=0}^{\infty} \frac{\lambda_\phi(p^{j})\lambda_\psi(p^{j+2})}{p^{j(1+2s)}}\Big) \\ \cdot \Big(1-\frac{\lambda_\phi(p)\lambda_\psi(p)}{p^{1+2s}} +\frac{\lambda_\phi(p^2)+\lambda_\psi(p^2)}{p^{2(1+2s)}} -\frac{\lambda_\phi(p)\lambda_\psi(p)}{p^{3(1+2s)}} +\frac{1}{p^{4(1+2s)}}\Big) \Big(1-\frac{1}{p^{2(1+2s)}}\Big)^{-1}.
\end{multline}
By the best known upper bound toward the Ramanujan and Selberg Conjectures, $\lambda_\phi(n)\ll n^{7/64+\varepsilon}$, \cite[Appendix 2]{Kim2003}, we know that $H_{\phi,\psi}(s,s_1,s_2)$ is absolutely convergent if
\[
  \Re(s_1+s_2) > -1, \quad
  \Re(s) > -\frac{25}{64}, \quad
  \Re(s+s_1) > -\frac{25}{64}, \quad
  \Re(s+s_2) > -\frac{25}{64}.
\]

Now we can estimate $\mathcal{S}_{\phi,\psi}$. If $\phi=\psi$, we have
\[
  L(s,\phi\times\phi) = \zeta(s) L(s,\Sym^2 \phi)
\]
has a simple pole at $s=1$ with residue $L(1,\Sym^2 \phi)$.
By shifting the contours in \eqref{eqn:S=int} to the left, say the vertical line $-\varepsilon+iv$, $v\in\mathbb{R}$, we obtain
\[
  \mathcal{S}_{\phi,\phi}  =  H_\phi \frac{L(1,\Sym^2\phi)}{\zeta(2)} \log t + O(1)
  \quad
  \textrm{with} \quad
  H_\phi=H_{\phi,\phi}(0,0,0).
\]

If $\phi\neq \psi$, then we know $L(s,\phi\times\psi)$ is entire for $s\in\mathbb{C}$, so by the same argument, we get
\[
  \mathcal{S}_{\phi,\psi}
  =  H_{\phi,\psi} \frac{L(1,\phi\times \psi)}{\zeta(2)} + O(T^{-\varepsilon})
  \quad \textrm{with} \quad
  H_{\phi,\psi} = H_{\phi,\psi}(0,0,0).
\]
Hence
\begin{equation}\label{eqn:D==}
  \mathcal{D}_{\phi,\psi} \sim
  \left\{ \begin{array}{ll}
  \frac{\log 2}{\zeta(2)} H_\phi L(1,\Sym^2\phi)
  \frac{\log T}{T}, &  \textrm{if $\phi=\psi$,}  \\
  \frac{\log 2}{\zeta(2)} H_{\phi,\psi} L(1,\phi\times \psi) \frac{1}{T}, &  \textrm{if $\phi\neq\psi$.}
  \end{array} \right.
\end{equation}

\section{Off-diagonal terms}

In this section, we will bound the off-diagonal terms $\mathcal{O}_{\phi,\psi}$ by applying the delta method.

\subsection{The initial cleaning}
First note that we can truncate the sums over $m$ and $n$ such that $mn\leq T^{2+\varepsilon}$, and the sums over $k$ and $\ell$ at $x^{1+\varepsilon}$, getting
\begin{multline*}
  \mathcal O_{\phi,\psi} =
  \frac{2}{T^2} \int_{0}^{\infty} w\Big(\frac{t}{T}\Big)
  \frac{1}{2\pi i} \int_{\varepsilon-iT^\varepsilon}^{\varepsilon+iT^\varepsilon} \Big(\frac{t^2}{\pi^2 } \Big)^s
   \sum_{k\leq x^{1+\varepsilon}} \sum_{\ell\leq x^{1+\varepsilon}}
   \frac{\alpha(k)\alpha(\ell)}{k\ell} e^{-k/x}e^{-\ell/x}
   \\   \cdot
   \sum_{h\neq0}  \sum_{\substack{mn\leq T^{2+\varepsilon} \\ km-\ell n=h}} \frac{\lambda_\phi(m)\lambda_\psi(n)}{(mn)^{1/2+s}} \Big(\frac{\ell n}{km}\Big)^{2it}
   e^{s^2} \frac{\dd s}{s} \dd t
   + O(T^{-A}).
\end{multline*}
Now we can apply a dyadic partition of unity to the sums over $m$ and $n$. That is, suppose $W(x)$ is a smooth, nonnegative function with support in $[1,2]$ such that $\sum_{M}W(x/M)=1$ for all $x\geq1$, where $M$ runs over a sequence of real numbers with $\#\{M: M\leq X\} \ll \log X$.
By changing variables, we have
\begin{multline*}
  \mathcal O_{\phi,\psi} =
  \frac{2}{T^2} \int_{0}^{\infty} w\Big(\frac{t}{T}\Big)
  \frac{1}{2\pi i} \int_{\varepsilon-iT^\varepsilon}^{\varepsilon+iT^\varepsilon} \Big(\frac{t^2}{\pi^2 } \Big)^s
  \sum_{d\leq x^{1+\varepsilon}}
   \sum_{\substack{k\leq x^{1+\varepsilon}/d \\ \ell\leq x^{1+\varepsilon}/d \\ (k,\ell)=1 }}
   \frac{\alpha(k d)\alpha(\ell d)}{k\ell d^2} e^{-kd/x}e^{-\ell d/x}
   \\   \cdot
   \sum_{\substack{M,N \\ MN\ll T^{2+\varepsilon}}}
   \sum_{h\neq0}  \sum_{\substack{m,n \\ km-\ell n=h}} \frac{\lambda_\phi(m)\lambda_\psi(n)}{(mn)^{1/2+s}} \Big(\frac{\ell n}{km}\Big)^{2it}
    W\Big(\frac{m}{M}\Big) W\Big(\frac{n}{N}\Big)
   e^{s^2} \frac{\dd s}{s} \dd t
   + O(T^{-A}).
\end{multline*}
Note that here we remove the condition $mn\leq T^{2+\varepsilon}$ with a cost of a negligible error.
We first consider the integral over $t$, that is,
\[
  \int_{0}^{\infty} w\Big(\frac{t}{T}\Big)
  \Big(\frac{\ell n}{km}\Big)^{2it} \dd t
  =
  T \int_{0}^{\infty} w(y)
  e^{2i y T \log(\frac{\ell n}{km})} \dd y.
\]
By repeated integration by parts, we may assume $|\log(\frac{\ell n}{km})|\ll UT^{-1+\varepsilon}$, that is,
$
  |\log(1-\frac{h}{km})| \ll UT^{-1+\varepsilon}.
$
Thus we obtain
\begin{equation}\label{eqn:MN&h}
  kM\asymp \ell N \gg T^{1-\varepsilon}U^{-1}, \quad  1\leq |h| \ll kMUT^{-1+\varepsilon} \asymp (k\ell MN)^{1/2} UT^{-1+\varepsilon}.
\end{equation}
Hence we have
\begin{multline}\label{eqn:O<<D}
  \mathcal O_{\phi,\psi} \ll T^{-2+\varepsilon}
  \int_{0}^{\infty} w\Big(\frac{t}{T}\Big) \int_{\varepsilon-iT^\varepsilon}^{\varepsilon+iT^\varepsilon} \sum_{d\leq x^{1+\varepsilon}}
   \sum_{\substack{k\leq x^{1+\varepsilon}/d \\ \ell\leq x^{1+\varepsilon}/d \\ (k,\ell)=1 }}
   \frac{1}{k\ell d^2}
  \sum_{\substack{M,N \\ MN\ll T^{2+\varepsilon}}} (MN)^{-1/2-\varepsilon}
   \\   \cdot
   \sum_{1\leq |h| \ll (k\ell MN)^{1/2} UT^{-1+\varepsilon}} \big| D_f(k,\ell;h) \big|
   |\dd s| \dd t
   + O(T^{-A}),
\end{multline}
where
\[
   D_f(k,\ell;h) = \sum_{\substack{m,n \\ km-\ell n=h}} \lambda_\phi(m)\lambda_\psi(n) f(km,\ell n),
\]
and
\begin{equation}\label{eqn:f}
  f(x,y) = \Big(\frac{k\ell MN}{xy}\Big)^{1/2+s} \Big(1-\frac{h}{x}\Big)^{2it} W\Big(\frac{x}{kM}\Big) W\Big(\frac{y}{\ell N}\Big).
\end{equation}

\subsection{Applying the delta method}


Our $\mathcal{O}_{\phi,\psi}$ is well-suited for application of the main result of Harcos \cite[Theorem 1]{Harcos2003}, which we reproduce here for completeness.
\begin{theorem}[Harcos] \label{thm:Harcos}
  Let $f$ be a smooth function on $(\mathbb{R}_{>0})^2$ satisfying
  \[
    x^i y^j f^{(i,j)}(x,y) \ll_{i,j} \Big(1+\frac{x}{X}\Big)^{-1}\Big(1+\frac{y}{Y}\Big)^{-1} P^{i+j},
  \]
  with some $P,X,Y\geq1$ for all $i,j\geq0$.
  Let $\lambda_\phi(m)$ (resp. $\lambda_\psi(n)$) be the normalized Fourier coefficients of a holomorphic or Maass cusp form $\phi$ (resp. $\psi$) of arbitrary level and nebentypus.  Define
  \begin{equation*}
    D_f(k,\ell;h)=\sum_{km\pm \ell n=h} \lambda_\phi(m)\lambda_\psi(n) f(km,\ell n),
  \end{equation*}
  where $k,\ell,h$ are positive integers.
  Then for coprime $k$ and $\ell$, we have
  \[
    D_f(k,\ell;h) \ll P^{11/10}(k\ell)^{-1/10}(X+Y)^{1/10}(XY)^{2/5+\varepsilon},
  \]
  where the implied constant depends only on $\varepsilon$ and the forms $\phi$, $\psi$.
\end{theorem}

We can apply this result to $f$ given by \eqref{eqn:f}, which satisfies the conditions with
\begin{equation}\label{eqn:XY&P}
  X=kM, \quad Y=\ell N, \quad P=UT^\varepsilon.
\end{equation}
Here one may use \eqref{eqn:MN&h}. Thus Theorem \ref{thm:Harcos} gives
\[
  D_f(k,\ell;h) \ll  T^\varepsilon U^{11/10} (k\ell)^{-1/10} (kM+\ell N)^{1/10}(k\ell MN)^{2/5+\varepsilon}.
\]
Hence by \eqref{eqn:O<<D}, we get
\begin{equation}\label{eqn:O<<}
  \mathcal{O}_{\phi,\psi} \ll T^{-1+\varepsilon} T^{-1/10}  U^{21/10} x^{17/10} \ll  T^{-1-\varepsilon},
\end{equation}
provided $U^{21/10} x^{17/10} \ll T^{1/10-\varepsilon}$.

\subsection{Proof of Theorem \ref{thm:QV}}
Now by \eqref{eqn:QV=L}, \eqref{eqn:I=}, \eqref{eqn:I=D+O}, \eqref{eqn:D==}, and \eqref{eqn:O<<} we have
\begin{equation*}
  \frac{1}{T} \int_{T}^{2T}
  \big| \mu_{t}(\phi) \big|^2  \dd t
  \sim
  \frac{\rho_\phi(1)^2 }{2^4} \Big|\Gamma\Big(\frac{1}{4}+\frac{it_\phi}{4}\Big)\Big|^4 L\Big(\frac{1}{2},\phi\Big)^2 \;
  \frac{2\log 2}{\zeta(2)} H_\phi L(1,\Sym^2\phi) \frac{\log T}{T} .
\end{equation*}
Recall that $H_\phi=H_{\phi,\phi}(0,0,0)$ with $H_{\phi,\psi}(s,s_1,s_2)$ defined by \eqref{eqn:H(s)}. Let
\begin{equation}\label{eqn:C}
  C(\phi) = (12 \cdot \log 2 ) H_\phi.
\end{equation}
By \eqref{eqn:V(phi)} and \eqref{eqn:RS}, we obtain
\begin{equation*}
  \frac{1}{T} \int_{T}^{2T}
  \big| \mu_{t}(\phi) \big|^2  \dd t
  \sim
  C(\phi)  L\Big(\frac{1}{2},\phi\Big)^2 \, V(\phi)  \frac{\log T}{T}.
\end{equation*}
And if $\phi\neq\psi$, then we have
\begin{multline*}
  \frac{1}{T} \int_{T}^{2T} \mu_t(\phi)\overline{\mu_t(\psi)} \dd t
  \sim
  \frac{\rho_\phi(1) \rho_\psi(1)\log 2}{2^3\zeta(2)}
  \Big|\Gamma\Big(\frac{1}{4}+\frac{it_\phi}{4}\Big)\Big|^2 \Big|\Gamma\Big(\frac{1}{4}+\frac{it_\psi}{4}\Big)\Big|^2
  \\
  \cdot L\Big(\frac{1}{2},\phi\Big)  L\Big(\frac{1}{2},\psi\Big)
  H_{\phi,\psi} L(1,\phi\times \psi) \frac{1}{T} .
\end{multline*}
This completes the proof of Theorem \ref{thm:QV}.

%
%
%

\section{The expected value}\label{sec:EV}

In this section, we will sketch the proof of Theorem \ref{thm:EV}.
Recall that $\phi$ is an even Hecke--Maass cusp form.
By Stirling's formula, for $t\sim T$ large and $t_\phi$ fixed, we have
\[
  \pi^{2it} \frac{\Gamma(\frac{1}{4}-\frac{it_\phi}{2}-it) \Gamma(\frac{1}{4}+\frac{it_\phi}{2}-it)} {\big|\Gamma(\frac{1}{2}+it)\big|^2}
   = e^{i\pi/4} t^{-1/2} e^{-2it\log \frac{t}{e\pi}} + O\Big(\frac{1}{T}\Big).
\]
By \eqref{eqn:mu_t=} and 
\eqref{eqn:SM-shortint}, we have
\begin{multline*}
  \mathbb{E}(\phi;T)
  = \frac{\rho_\phi(1)}{4}  \big|\Gamma(\frac{1}{4}+\frac{it_\phi}{2})\big|^2 L\Big(\frac{1}{2},\phi\Big) \\ \cdot
  \frac{1}{T}\int_{T}^{2T} \pi^{2it} \frac{\Gamma(\frac{1}{4}-\frac{it_\phi}{2}-it) \Gamma(\frac{1}{4}+\frac{it_\phi}{2}-it)} {\big|\Gamma(\frac{1}{2}+it)\big|^2}
  \frac{L\big(\frac{1}{2}-2it,\phi\big)}{|\zeta(1+2it)|^2} \dd t \\
  = \frac{e^{i\pi/4}\rho_\phi(1)}{4}  \big|\Gamma(\frac{1}{4}+\frac{it_\phi}{2})\big|^2 L\Big(\frac{1}{2},\phi\Big)
  \frac{1}{T}\int_{T}^{2T} t^{-1/2}
  e^{-2it\log \frac{t}{e\pi}} \frac{L\big(\frac{1}{2}-2it,\phi\big)}{|\zeta(1+2it)|^2}
  \dd t + O(T^{-1}).
\end{multline*}
By the same method as in Lemma \ref{lemma:AFE-zeta}, we have
\[
  \frac{1}{\zeta(1+2it)} = \sum_{k\leq x^{1+\varepsilon} } \frac{\mu(k)}{k^{1+2it}} e^{-k/x} + O(e^{-(\log T)^{1/5}}),
\]
for any $T^\varepsilon\ll x \ll T^B$ and  $t\asymp T$ large.
The contribution from the above error term is
\[
  \ll_\phi T^{-3/2} e^{-(\log T)^{1/5}}
  \int_{T}^{2T}  \Big|L\Big(\frac{1}{2}-2it,\phi\Big)\Big|^2 \dd t
  \ll_\phi T^{-1/2} e^{-(\log T)^{1/6}} .
\]
To prove Theorem \ref{thm:EV}, it suffices to consider
\begin{equation}\label{eqn:Laver}
  \frac{1}{T}\int_{T}^{2T} t^{-1/2}  e^{-2it\log \frac{t}{e\pi}}
  \Big| \sum_{k\leq x^{1+\varepsilon} } \frac{\mu(k)}{k^{1+2it}} e^{-k/x} \Big|^2 L\Big(\frac{1}{2}-2it,\phi\Big) \dd t.
\end{equation}
We may want to replace it by a smooth average. As in \S \ref{sec:QVsetup}, we consider
\begin{equation}\label{eqn:Laver-v}
  \frac{1}{T^{3/2}}\int_{0}^{\infty} v\Big(\frac{t}{T}\Big) e^{-2it\log \frac{t}{e\pi}}
  \Big| \sum_{k\leq x^{1+\varepsilon} } \frac{\mu(k)}{k^{1+2it}} e^{-k/x} \Big|^2 L\Big(\frac{1}{2}-2it,\phi\Big) \dd t,
\end{equation}
where $v(y)=v_U(y)$ is a smooth function with support in the interval $[1-1/U,2+1/U]$, $U\geq2$, such that $v_U^{(j)}(y) \ll U^j$ for $j\geq0$ and $v_U(y)=y^{-1/2}$ for $y\in[1+1/U,2-1/U]$.
Here we assume $T^\varepsilon\leq U \leq T^{1/3}$.
By \eqref{eqn:SM-shortint}, we know that the difference between \eqref{eqn:Laver} and \eqref{eqn:Laver-v} is $O_\phi(T^{-1/2+\varepsilon}U^{-1})$.
And as in Lemma \ref{lemma:AFE-Lfcn} (cf. \cite[\S5.2]{iwaniec2004analytic}), we have
\begin{equation*}
  L\Big(\frac{1}{2}-2it,\phi\Big) \\
    = \sum_{n\leq T^{1+\varepsilon}} \frac{\lambda_\phi(n)}{n^{1/2-2it}} W_{1}^-(n;t)
    + \frac{1}{i}
    \Big(\frac{t}{e\pi}\Big)^{4it}
    \sum_{n\leq T^{1+\varepsilon}} \frac{\lambda_\phi(n)}{n^{1/2+2it}} W_{1}^+(n;t)
    + O(T^{-1/2+\varepsilon}),
\end{equation*}
where
\[
  W_1^{\pm}(y;t)
  = \frac{1}{2\pi i} \int_{\varepsilon-iT^\varepsilon}^{\varepsilon+iT^\varepsilon} \Big( \frac{t}{\pi y} \Big)^s e^{\pm i \frac{\pi}{2} s} e^{s^2} \frac{\dd s}{s}
  .
\]
The above error term contributes
$O_\phi( T^{-1+\varepsilon} )$.
To prove Theorem \ref{thm:EV}, it suffices to consider
\begin{multline}\label{eqn:Lmain}
  \frac{1}{T^{3/2}}\int_{0}^{\infty} v\Big(\frac{t}{T}\Big)  e^{\pm 2it\log \frac{t}{e\pi}}
  \Big| \sum_{k\leq x^{1+\varepsilon} } \frac{\mu(k)}{k^{1+2it}} e^{-k/x} \Big|^2 \sum_{n\leq T^{1+\varepsilon}} \frac{\lambda_\phi(n)}{n^{1/2\pm2it}} W_{1}^\pm(n;t) \dd t \\
  = \frac{1}{2\pi i} \int_{\varepsilon-iT^\varepsilon}^{\varepsilon+iT^\varepsilon}
  \frac{1}{T^{3/2}}
  \sum_{k\leq x^{1+\varepsilon} } \frac{\mu(k)}{k} e^{-k/x}
  \sum_{\ell\leq x^{1+\varepsilon} } \frac{\mu(\ell)}{\ell} e^{-\ell/x}
  \sum_{n\leq x^2 T^{1+\varepsilon}} \frac{\lambda_\phi(n)}{n^{1/2+s}} \\
  \cdot \Big( \int_{0}^{\infty} v\Big(\frac{t}{T}\Big)  t^s e^{\pm 2it\log \frac{t}{e\pi}}
  \Big(\frac{\ell n}{k}\Big)^{\mp2it}  \dd t  \Big)
  \pi^{-s} e^{\pm i \frac{\pi}{2} s}  e^{s^2} \frac{\dd s}{s} + O(T^{-A}).
\end{multline}

We first deal with the $t$-integral
\[
  \int_{0}^{\infty} v\Big(\frac{t}{T}\Big)  t^s e^{\pm 2it\log \frac{t}{e\pi}}
  \Big(\frac{\ell n}{k}\Big)^{\mp2it}  \dd t
  = \int_{0}^{\infty} w_1(t)
  e^{i h_1(t) } \dd t,
\]
where
\[
  w_1(t)=v\Big(\frac{t}{T}\Big)  t^s
  \quad \textrm{and} \quad
  h_1(t)= \mp  2t \log(\frac{e\pi \ell n}{k t}).
\]
If $\frac{\pi \ell n}{k}<\frac{T}{2}$ or $\frac{\pi \ell n}{k}>3T$, then by repeated integration by parts, we have
(see e.g. \cite[Lemma 8.1]{blomer2013distribution})
\[
  \int_{0}^{\infty} v\Big(\frac{t}{T}\Big)  t^s e^{\pm 2it\log \frac{t}{e\pi}}
  \Big(\frac{\ell n}{k}\Big)^{\mp2it}  \dd t
  \ll T^{-A},
\]
for any $A>0$, since for $t\in \supp w_1 \subset [\frac{3}{4}T,\frac{5}{2}T]$ we have
\begin{equation}\label{eqn:w1j}
  w_1^{(j)}(t) \ll_j T^\varepsilon \Big(\frac{T}{U}\Big)^{-j}
   \quad \textrm{for $j\geq0$},
\end{equation}
\begin{equation}\label{eqn:h1>>}
  |h_1'(t)|=|2\log\frac{\pi \ell n}{kt}|\gg 1,
\end{equation}
\begin{equation}\label{eqn:h1j=}
  h_1^{(j)}(t) \asymp_j T^{1-j} \quad \textrm{for $j\geq2$}.
\end{equation}
This contributes a negligible error to \eqref{eqn:Lmain}.

If $\frac{T}{2} \leq \frac{\pi \ell n}{k} \leq 3T$, then
we want to apply the stationary phase method (see e.g. \cite[Proposition 8.2]{blomer2013distribution}).
Note that $\supp w_1 \subset [\frac{T}{4},4T]$.
Let $t_0=\frac{\pi \ell n}{k}$. Note that $h_1'(t_0)=0$.
Since for $t\in[\frac{T}{4},4T]$ we have \eqref{eqn:w1j}, \eqref{eqn:h1j=}, and
$
  h_1'(t)\ll 1,
$
we obtain
\[
  \int_{0}^{\infty} v\Big(\frac{t}{T}\Big)  t^s e^{\pm 2it\log \frac{t}{e\pi}}
  \Big(\frac{\ell n}{k}\Big)^{\mp2it}  \dd t
  = \sqrt{\frac{\pi \ell n}{k}} e^{\mp 2\pi i \frac{\ell n}{k}} w_{\pm}\Big(\frac{\pi \ell n}{k}\Big) + O(T^{-A}),
\]
where $w_{\pm}$ is a smooth function depending on $w_1$ and $h_1$, satisfying that
\begin{equation}\label{eqn:p_pm}
  \frac{\dd^j}{\dd x^j}w_{\pm}(x) \ll T^\varepsilon \Big(\frac{T}{U}\Big)^{-j}
  \quad \textrm{and} \quad
  \supp w_{\pm} \subseteq \supp w_1 .
\end{equation}
The contribution from the above error term to \eqref{eqn:Lmain} is negligible again.
Hence to bound \eqref{eqn:Lmain}, we need to consider
\[
  \sum_{n} \lambda_\phi(n)
  e^{\mp 2\pi i \frac{\ell n}{k}} n^{-s} w_{\pm}\Big(\frac{\pi \ell n}{k}\Big).
\]
By splitting the $n$-sum in to residue classes mod $k$, and detecting the summation congruence condition by primitive additive characters, we obtain
\begin{multline*}
  \sum_{n\geq1} \lambda_\phi(n)
  e^{\mp 2\pi i \frac{\ell n}{k}} n^{-s} w_{\pm}\Big(\frac{\pi \ell n}{k}\Big) \\
  =
  \sum_{a \; (\rmmod k)} e\Big(\mp\frac{\ell a}{k}\Big) \frac{1}{k} \sum_{q\mid k} \underset{b\; (\rmmod q)}{{\sum}^*} e\Big(\frac{-ab}{q}\Big)
  \sum_{n\geq1} \lambda_\phi(n) e\Big(\frac{bn}{q}\Big)
   n^{-s} w_{\pm}\Big(\frac{\pi \ell n}{k}\Big).
\end{multline*}
Hence it suffices to consider
\[
  \Sigma := \sum_{n\geq1} \lambda_\phi(n) e\Big(\frac{bn}{q}\Big)
   n^{-s} w\Big(\frac{\pi \ell n}{k}\Big),
\]
where $(b,q)=1$ and $w=w_+$ or $w_-$.
Now we are ready to apply the Voronoi summation formula \cite[Theorem 2]{meurman1988}, getting
\begin{multline}\label{eqn:Sigma=sums}
  \Sigma
   = \frac{-i}{q} \sum_{n\geq1} \lambda_\phi(n) e\Big(-\frac{\bar{b}n}{q}\Big) \int_{0}^{\infty} J_{2it_\phi}^+(\frac{4\pi \sqrt{ny}}{q}) \; y^{-s} w\Big(\frac{\pi \ell y}{k}\Big) \dd y \\
  + \frac{1}{q} \sum_{n\geq1} \lambda_\phi(n) e\Big(\frac{\bar{b}n}{q}\Big) \int_{0}^{\infty} K_{2it_\phi}^+\Big(\frac{4\pi \sqrt{ny}}{q}\Big) \; y^{-s} w\Big(\frac{\pi \ell y}{k}\Big) \dd y,
\end{multline}
where
\[
  J_{\nu}^+(z) := \frac{-\pi}{\sinh (\pi \nu/2)} (J_{\nu}(z) - J_{-\nu}(z)), \quad
  K_{\nu}^+(z) := 4 \cosh(\pi \nu/2) K_{\nu}(z).
\]
Note that $q\leq k\leq x^{1+\varepsilon}$, and by the properties of $w_{\pm}$ \eqref{eqn:p_pm} we know that $y\asymp \frac{k T}{\ell}$. By asymptotic evaluations
regarding the Bessel functions
(cf. Hough \cite[eq. (3.5) and (3.6)]{hough2012zero})
\begin{equation}\label{eqn:J&K}
  \begin{split}
     J_\nu^+(z) & = -\sqrt{\frac{2\pi}{z}}\sin(z-\pi/4) \left[1-\frac{16\nu^4-40\nu^2+9}{128z^2}\right] \\
                & \hskip 90pt - \pi \cos(z-\pi/4)\frac{\nu^2-1/4}{2z} \; + \; O\left(\frac{1+|\nu|^6}{z^3}\right), \\
     K_\nu(x)   & = \sqrt{\frac{\pi}{2z}} e^{-z} \left[1+ O\left(\frac{1+|\nu|^2}{z}\right)\right],
  \end{split}
\end{equation}
we know that the second sum in \eqref{eqn:Sigma=sums} is negligible provided $x\leq T^{1/2-\varepsilon}$.
Note that here we use the fact that by the Rankin--Selberg method we have
\[
  \sum_{n\leq N} \lambda_\phi(n)^2 \sim c_{\phi} N.
\]
The contribution from the error term of $J_{\nu}^+$ in \eqref{eqn:J&K} to \eqref{eqn:Lmain} 
is bounded by
\[
  O_\phi\Big(T^{-2+\varepsilon}x^{3+\varepsilon} \sum_n \frac{|\lambda_\phi(n)|}{n^{3/2}}\Big)
  = O_\phi\Big(T^{-2+\varepsilon}x^{3+\varepsilon} \Big)
  = O_\phi\Big(T^{-1+\varepsilon} U^{-2} \Big),
\]
if $x^{3/2}U\leq T^{1/2-\varepsilon}$.
Now we deal with the contribution from the main terms of $J_{\nu}^+$ in \eqref{eqn:J&K}.
Note that we can rearrange the main terms of $J_{\nu}^+(z)$ as $\sum_\pm P_\nu^\pm(z) e^{\pm iz}$ with
\[
  P_\nu^\pm(z)
   = \pm \sqrt{\frac{2\pi}{z}}
     \left(1-\frac{16\nu^4-40\nu^2+9}{128z^2}\right)
     \frac{e^{\mp i\pi/4}}{2i}
     - \pi \left(\frac{\nu^2-1/4}{2z}\right)
     \frac{e^{\mp i\pi/4}}{2}.
\]
It suffices to consider
\[
  \frac{1}{q} \sum_{n\geq1} \lambda_\phi(n) e\Big(-\frac{\bar{b}n}{q}\Big) \int_{0}^{\infty} P_{2it_\phi}^\pm\Big(\frac{4\pi \sqrt{ny}}{q}\Big) e^{\pm i \frac{4\pi \sqrt{ny}}{q}} \; y^{-s} w\Big(\frac{\pi \ell y}{k}\Big) \dd y.
\]
For the above $y$-integral, we define
\[
  w_2(y) = P_{2it_\phi}^\pm\Big(\frac{4\pi \sqrt{ny}}{q}\Big) \; y^{-s} w\Big(\frac{\pi \ell y}{k}\Big)
  \quad \textrm{and} \quad
  h_2(y) = \pm  \frac{4\pi \sqrt{ny}}{q}.
\]
Note that $\supp w_2 \subset [\frac{3}{4}\frac{kT}{\pi \ell},\frac{5}{2}\frac{kT}{\pi \ell}]$.
Assume $x\leq U^{1-\varepsilon}$.
For $y\in [\frac{3}{4}\frac{kT}{\pi \ell},\frac{5}{2}\frac{kT}{\pi \ell}]$ and $\ell \leq x^{1+\varepsilon}$, by \eqref{eqn:p_pm}, we have
\begin{equation}\label{eqn:w2j}
  w_2^{(j)}(y) \ll_j \Big(\frac{\sqrt{ny}}{q}\Big)^{-1/2} \Big( y^{-j} + T^\varepsilon \Big(\frac{T}{U}\Big)^{-j}\Big)
  \ll_j \Big(\frac{\sqrt{nkT}}{q\sqrt{\ell}}\Big)^{-1/2} T^\varepsilon \Big(\frac{T}{U}\Big)^{-j}
   \quad \textrm{for $j\geq0$},
\end{equation}
\begin{equation}\label{eqn:h2j=}
  h_2^{(j)}(y) \asymp_j \frac{\sqrt{n}}{q} \Big(\frac{kT}{\ell}\Big)^{1/2-j} \quad \textrm{for $j\geq1$}.
\end{equation}
Hence, by repeated integration by parts again (cf. \cite[Lemma 8.1]{blomer2013distribution}), we get
\[
  \int_{0}^{\infty} \Big(\frac{4\pi \sqrt{ny}}{q}\Big)^{-1/2} e^{\pm i \frac{4\pi \sqrt{ny}}{q}} \; y^{-s} w\Big(\frac{\pi \ell y}{k}\Big) \dd y
  \ll n^{-A} T^{-A},
\]
provided $x^{3/2}U\leq T^{1/2-\varepsilon}$.
Hence \eqref{eqn:Lmain} is bounded by
\[
  O_\phi(T^{-1+\varepsilon}U^{-2}),
\]
if $x^{3/2}U\leq T^{1/2-\varepsilon}$.
Now by taking $U=T^{1/3}$ and $x = T^{1/10}$ for example, we prove
$$
  \mathbb{E}(\phi;T) = O_\phi( T^{-1/2} e^{-(\log T)^{1/6}})=o(T^{-1/2}).
$$

\appendix

\section{Dirichlet series} \label{app:DS}

In this appendix, we prove \eqref{eqn:D=LH}.
As in \cite[\S3]{HughesYoung2010}, for $\Re(s)$ large enough, we have
\begin{multline*}
  \sum_{n} \frac{\lambda_\phi(\ell n)\lambda_\psi(k n)}{n^{1+2s}}
  =
  \Big(\prod_{(p,k\ell)=1} \sum_{j=0}^{\infty} \frac{\lambda_\phi(p^j)\lambda_\psi(p^j)}{p^{j(1+2s)}}\Big) \\
  \cdot \Big(\prod_{p|\ell} \sum_{j=0}^{\infty} \frac{\lambda_\phi(p^{j+\ell_p})\lambda_\psi(p^j)}{p^{j(1+2s)}}\Big)
  \Big(\prod_{p|k} \sum_{j=0}^{\infty} \frac{\lambda_\phi(p^j)\lambda_\psi(p^{j+k_p})}{p^{j(1+2s)}}\Big),
\end{multline*}
where
$p^{k_p}\| k$ and $p^{\ell_p}\| \ell$.
Note that
\[
  \prod_{p} \sum_{j=0}^{\infty} \frac{\lambda_\phi(p^j)\lambda_\psi(p^j)}{p^{j(1+2s)}}
  = \sum_{n} \frac{\lambda_\phi(n)\lambda_\psi(n)}{n^{1+2s}}
  = \frac{L(1+2s,\phi\times \psi)}{\zeta(2+4s)}.
\]
We have
\begin{multline*}
  \sum_{n} \frac{\lambda_\phi(k n)\lambda_\psi(\ell n)}{n^{1+2s}}
  = \frac{L(1+2s,\phi\times \psi)}{\zeta(2+4s)}
  \Big(\prod_{p|\ell} \sum_{j=0}^{\infty} \frac{\lambda_\phi(p^{j+\ell_p})\lambda_\psi(p^j)}{p^{j(1+2s)}}\Big)
    \\ \cdot
  \Big(\prod_{p|k} \sum_{j=0}^{\infty} \frac{\lambda_\phi(p^j)\lambda_\psi(p^{j+k_p})}{p^{j(1+2s)}}\Big) \\
  \cdot \Big(\prod_{p|k\ell} \Big(1-\frac{\lambda_\phi(p)\lambda_\psi(p)}{p^{1+2s}} +\frac{\lambda_\phi(p^2)+\lambda_\psi(p^2)}{p^{2(1+2s)}} -\frac{\lambda_\phi(p)\lambda_\psi(p)}{p^{3(1+2s)}} +\frac{1}{p^{4(1+2s)}}\Big) \Big(1-\frac{1}{p^{2(1+2s)}}\Big)^{-1}\Big).
\end{multline*}

Similarly, for $\Re(s_1+s_2)$ large enough, we have
\begin{multline*}
  \sum_{d} \frac{\alpha(k d)\alpha(\ell d)}{d^{2+s_1+s_2}}
  =
  \Big(\prod_{(p,k\ell)=1} \sum_{j=0}^{\infty} \frac{\alpha(p^j)\alpha(p^j)}{p^{j(2+s_1+s_2)}}\Big)
  \\ \cdot
  \Big(\prod_{p|k} \sum_{j=0}^{\infty} \frac{\alpha(p^{j+k_p})\alpha(p^j)}{p^{j(2+s_1+s_2)}}\Big)
  \Big(\prod_{p|\ell} \sum_{j=0}^{\infty} \frac{\alpha(p^j)\alpha(p^{j+\ell_p})}{p^{j(2+s_1+s_2)}}\Big).
\end{multline*}
Note that
\[
  \alpha(p^j) = \left\{ \begin{array}{ll}
                          1, & \textrm{if $j=0$ or $j=2$,} \\
                          -2, & \textrm{if $j=1$,} \\
                          0, & \textrm{if $j\geq3$,} \\
                        \end{array}  \right.
\]
and then
\[
  \prod_{p} \sum_{j=0}^{\infty} \frac{\alpha(p^j)\alpha(p^j)}{p^{j(2+s_1+s_2)}}
  = \prod_p  \Big( 1 + \frac{4}{p^{2+s_1+s_2}} + \frac{1}{p^{2(2+s_1+s_2)}} \Big).
\]
Hence
\begin{multline*}
  \sum_{d} \frac{\alpha(k d)\alpha(\ell d)}{d^{2+s_1+s_2}}
  =
  \Big(\prod_p  \Big( 1 + \frac{4}{p^{2+s_1+s_2}} + \frac{1}{p^{2(2+s_1+s_2)}} \Big)\Big) 
  \Big(\prod_{p|k\ell}  \Big( 1 + \frac{4}{p^{2+s_1+s_2}} + \frac{1}{p^{2(2+s_1+s_2)}} \Big)^{-1} \Big) \\
  \cdot \Big(\prod_{p|k} \sum_{j=0}^{\infty} \frac{\alpha(p^{j+k_p})\alpha(p^j)}{p^{j(2+s_1+s_2)}}\Big)
  \Big(\prod_{p|\ell} \sum_{j=0}^{\infty} \frac{\alpha(p^j)\alpha(p^{j+\ell_p})}{p^{j(2+s_1+s_2)}}\Big).
\end{multline*}
Note that we have $0\leq k_p,\ell_p \leq 2$ and $\min(k_p,\ell_p)=0$, we have
\begin{multline*}
  D_{\phi,\psi}(s,s_1,s_2) = \frac{L(1+2s,\phi\times \psi)}{\zeta(2+4s)} \Big(\prod_p  \Big( 1 + \frac{4}{p^{2+s_1+s_2}} + \frac{1}{p^{2(2+s_1+s_2)}} \Big)\Big) \\
  \cdot \Big( \prod_p \Big( 1 + \frac{A_{\phi,\psi,1}(p)}{p^{3/2+s+s_1}} + \frac{A_{\phi,\psi,2}(p)}{p^{2(3/2+s+s_1)}} + \frac{A_{\psi,\phi,1}(p)}{p^{3/2+s+s_2}} +
  \frac{A_{\psi,\phi,2}(p)}{p^{2(3/2+s+s_2)}} \Big) \Big),
\end{multline*}
where
\begin{multline*}
  A_{\phi,\psi,1}(p) = \Big(\sum_{j=0}^{\infty} \frac{\lambda_\phi(p^{j})\lambda_\psi(p^{j+1})}{p^{j(1+2s)}}\Big) \Big(-2-\frac{2}{p^{2+s_1+s_2}}\Big)   \Big( 1 + \frac{4}{p^{2+s_1+s_2}} + \frac{1}{p^{2(2+s_1+s_2)}} \Big)^{-1}
  \\ \cdot \Big(1-\frac{\lambda_\phi(p)\lambda_\psi(p)}{p^{1+2s}} +\frac{\lambda_\phi(p^2)+\lambda_\psi(p^2)}{p^{2(1+2s)}} -\frac{\lambda_\phi(p)\lambda_\psi(p)}{p^{3(1+2s)}} +\frac{1}{p^{4(1+2s)}}\Big) \Big(1-\frac{1}{p^{2(1+2s)}}\Big)^{-1},
\end{multline*}
\begin{multline*}
  A_{\phi,\psi,2}(p) = \Big(\sum_{j=0}^{\infty} \frac{\lambda_\phi(p^{j})\lambda_\psi(p^{j+2})}{p^{j(1+2s)}}\Big) \Big( 1 + \frac{4}{p^{2+s_1+s_2}} + \frac{1}{p^{2(2+s_1+s_2)}} \Big)^{-1}
  \\ \cdot \Big(1-\frac{\lambda_\phi(p)\lambda_\psi(p)}{p^{1+2s}} +\frac{\lambda_\phi(p^2)+\lambda_\psi(p^2)}{p^{2(1+2s)}} -\frac{\lambda_\phi(p)\lambda_\psi(p)}{p^{3(1+2s)}} +\frac{1}{p^{4(1+2s)}}\Big)  \Big(1-\frac{1}{p^{2(1+2s)}}\Big)^{-1}.
\end{multline*}
By rearranging the factors, we prove \eqref{eqn:D=LH}.

\section*{Acknowledgements}

The author is indebted to Prof. Ze\'{e}v Rudnick for teaching him the quantum variance and many useful discussions.
It is impossible to finish this work without his encouragement and help.
He wants to thank Prof. Paul Nelson for suggesting to think about Eisenstein series, and Peter Humphries for many helpful comments, especially for the local computations.


\end{document}